\documentclass[12pt]{article}
\usepackage{graphicx}
\usepackage{amsthm}
\usepackage{amssymb}
\usepackage{amsmath}
\usepackage{amscd}
\usepackage{bbm}
\usepackage{float}
\usepackage{booktabs}
\numberwithin{equation}{section}

\newtheorem{theorem}{Theorem}[section]
\newtheorem{lemma}{Lemma}[section]
\newtheorem{remark}{Remark}[section]

\newtheorem{proposition}{Proposition}[section]

\usepackage[numbers,sort&compress]{natbib}

\numberwithin{figure}{section}
\numberwithin{table}{section}

\allowdisplaybreaks[2]
\newcommand{\ep}{\mathrm{e}}
\usepackage{pgfplots}
\usepackage{enumitem} 
\newtheoremstyle{mystep}
{3pt}   
{3pt}   
{\normalfont} 
{0pt}   
{\bfseries} 
{.}    
{5pt plus 1pt minus 1pt} 
{}     
\theoremstyle{mystep}

\newcommand\btd{\raise 2pt \hbox{$\hat\bigtriangledown$}\hskip 1.5pt}
\newcommand\bt{\raise 2pt \hbox{$\bigtriangledown$}\hskip 1.5pt}

\usepackage{mathrsfs}
\begin{document}
	
	\title{  Hydroelastic waves near shear flows}
	\date{}
	\author{Jian Li$^a$,~~Shaojie Yang$ ^{b}$\thanks{Corresponding author: shaojieyang@kust.edu.cn (Shaojie Yang)\newline$~~~~~~~$Email: jianli\_jakura@163.com(Jian Li); shaojieyang@kust.edu.cn (Shaojie Yang)  } \\~\\
		\small$^a$ Institute for Advanced Study, Shenzhen University, \\
		\small Shenzhen, Guangdong 518060, China\\~\\
		\small$^b$ Department of  Mathematics,~~Kunming University of Science and Technology,  \\
		\small Kunming, Yunnan 650500, China
	}
		\maketitle
	\begin{abstract}
	In this paper,  we prove that symmetric, doubly-periodic, three-dimensional hydroelastic waves of 
	small amplitude, bifurcating from a non-uniform shear flow, are necessarily 
	two-dimensional to leading order.  A detailed analysis of the solvability condition at 
	quadratic order reveals that all Fourier coefficients with transverse dependence vanish 
	identically. The dimensional reduction is controlled entirely by the vertical structure 
	of the shear flow and is robust under changes to the dynamic boundary condition.\\
	
	
	\end{abstract}
\noindent\rule{15.5cm}{0.6pt}	
%\newpage	
%\tableofcontents		

\section{Introduction}
%\subsection{The Hydroelastic Waves}
%\hspace{1.5em}
Hydroelastic waves arise from the interaction between an inviscid fluid and a thin elastic sheet floating on its surface. They are relevant to a wide range of applications, including very large floating structures, flexible risers, and the use of sea ice as roads and landing strips \cite{r2,r3,r4,r5}. A thorough understanding of their qualitative behavior is therefore of both theoretical and practical importance.
Several nonlinear models have been developed to describe the elastic layer. The Kirchhoff--Love plate model, in particular, has been widely used to study periodic and solitary waves \cite{Forbes1986,Forbes1988,VandenBroeck2011,Milewski2011,Milewski2013} as well as moving load problems \cite{Bonnefoy2009,Parau2002}. A more general, fully nonlinear formulation based on the Cosserat theory of hyperelastic shells was introduced by Toland \cite{Toland2007,Toland2008} and subsequently extended to three dimensions \cite{PT2011}. The existence of hydroelastic solitary waves propagating at the ice-fluid interface has been established in \cite{r10,r11}. Utilizing the global bifurcation approach, subsequent studies \cite{r12,r13} demonstrated the existence of periodic hydroelastic waves, reformulating the problem as a vortex sheet to accommodate configurations where the elastic plate separates two immiscible irrotational fluid layers of distinct densities.  The well-posedness of the associated initial-value problem was investigated in \cite{r15,R1}.  The stability characteristics of deep-water hydroelastic waves were introduced in \cite{r19}. The exact energy relation was derived in \cite{zamp}. For numerical studies of hydroelastic waves, we refer to the recent work \cite{Gao2014,Gao2016,Guyenne2012,Guyenne2014}.

%In Ref.\cite{r10,r11}, the authors present an existence theory for hydroelastic solitary waves at the interface between a thin ice sheet and an ideal fluid. In Ref.\cite{r12,r13}, the existence of periodic hydroelastic waves which may posses a multi-valued height between two superposed irrotational fluid layers with positive densities separated by an elastic plate was shown via global bifurcation theorem, the analysis being based on the reformulation of the problem as a vortex sheet problem. The well-posedness  for the initial-value problem for hydroelastic waves was introduced in \cite{r15,R1}. Stability of hydroelastic waves in deep water was studied in \cite{r19}.
% this framework has since been adopted in numerous investigations \cite{Gao2014,Gao2016,Guyenne2012,Guyenne2014}.

\subsection{The governing equations}

We consider the problem of steady three-dimensional doubly-periodic hydroelastic waves 
propagating on an incompressible fluid of finite depth, beneath an elastic thin shell 
satisfying Kirchhoff's hypothesis. The fluid occupies the domain
\begin{align*}
	\Omega^\eta = \{ (x, y, z) \in \mathbb{R}^3 : -d < z < \eta(x, y) \},
\end{align*}
with a flat rigid bottom at \(z = -d\) and a free surface \(S^\eta = \{z = \eta(x, y)\}\) 
which coincides with the elastic shell. The flow is assumed to be steady in a frame moving 
with the wave, with velocity field \(\mathbf{u} = (u, v, w)\) and pressure \(p\) satisfying 
the incompressible Euler equations
\begin{align}
	(\mathbf{u} \cdot \nabla) \mathbf{u} + \nabla p + g \mathbf{e}_3 &= 0 \quad \text{in } \Omega^\eta, \label{eq:momentum}\\
	\nabla \cdot \mathbf{u} &= 0 \quad \text{in } \Omega^\eta,
\end{align}
and the kinematic boundary conditions
\begin{align}
	w &= 0 \quad \text{on }  z = -d, \label{eq:bottom}\\
	w &= u \partial_x \eta + v \partial_y \eta \quad \text{on }  z = \eta(x, y). \label{eq:kinematic}
\end{align}

The dynamic boundary condition incorporates the bending energy of the elastic shell. 
Following Plotnikov and Toland \cite{PT2011}, we consider a hyperelastic shell whose 
bending energy density \(W_b(H)\) depends only on the mean curvature \(H\). 
In typical applications, such as the modeling of sea ice or very large floating structures, 
the mass per unit area of the elastic sheet is small compared to the mass of the displaced 
fluid, so that the inertial forces of the shell may be neglected; mathematically, the 
reference density \(\rho_0 = 0\) is set to zero, making the elastic response purely 
quasi-static. Let the unit normal vector \(\mathbf{n}\) to 
the free surface be directed from the fluid into the elastic shell (i.e., pointing towards 
the positive \(z\) half-space). Then the fluid exerts a force \(-p\mathbf{n}\) on the shell, 
while the shell exerts a reaction force \(\Psi_n\mathbf{n}\) on the fluid, where \(\Psi_n\) 
is the normal stress density of the shell \textit{relative to the reference area element} 
\(\sqrt{a}\,dX\) (i.e., \(\Psi_n\) is defined exactly as in Plotnikov and Toland 
\cite{PT2011}, Eq.\ (1.10a)). Balance of normal forces across the interface yields the 
dynamic boundary condition
\begin{align}
	p + \frac{1}{\sqrt{a}} \Psi_n = 0 \quad \text{on } S^\eta, \label{eq:dynamic}
\end{align}
where \(a = \det(a_{ij}) = 1 + |\nabla\eta|^2\)
and \(\Psi_n\) is given by
\begin{align}
	\Psi_n = \frac{\sqrt{a}}{2} \Delta_\Sigma W_b'(H) + 2\sqrt{a}\,H\big(H W_b'(H) - W_b(H)\big) - \sqrt{a}\,K W_b'(H). \label{eq:Psi_n}
\end{align}
Here \(H\) and \(K\) are the mean and Gaussian curvatures of the surface, and 
\(\Delta_\Sigma\) is the Laplace--Beltrami operator on the surface. In Monge coordinates, 
with the parametrisation \((x, y) \mapsto (x, y, \eta(x, y))\), these geometric quantities 
are given explicitly by
\begin{align}
	H
	&=
	\frac12\nabla\cdot
	\left(
	\frac{\nabla\eta}{\sqrt{1+|\nabla\eta|^2}}
	\right),
	\label{eq:H}\\[6pt]
	K
	&=
	\frac{\partial_{xx}\eta\,\partial_{yy}\eta-(\partial_{xy}\eta)^2}
	{(1+|\nabla\eta|^2)^2},\\[6pt]
	\Delta_\Sigma f
	&=
	\frac1{\sqrt a}
	\partial_i
	\left(
	\sqrt a\,a^{ij}\partial_j f
	\right),
	\qquad
	a^{ij}
	=
	\delta_{ij}
	-
	\frac{\eta_i\eta_j}{1+|\nabla\eta|^2},
	\label{eq:LaplaceBeltrami}
\end{align}
where repeated indices \(i,j\in\{1,2\}\) are summed and
\(a=1+|\nabla\eta|^2\).

\begin{remark}
	\normalfont
 The factor 
	\(1/\sqrt{a}\) in \eqref{eq:dynamic} converts the reference-area density \(\Psi_n\) to a 
	physical-area density, so that \(p + (1/\sqrt{a})\Psi_n = 0\) is the correct balance of 
	forces per unit deformed area. Up to the sign convention in the definition of the normal 
	vector, this is equivalent to the dynamic condition \(p = (1/\sqrt{a})\Psi_n\) appearing 
	in Plotnikov and Toland \cite{PT2011}, Eq.\ (2.2a).
\end{remark}

\subsection{Main result}
We are interested in waves that are doubly periodic with respect to a given lattice and 
satisfy the \((+)\)-symmetry condition introduced in~\cite{SVW2024}.
More specifically, let
\[
\Lambda = \lambda_1 \mathbb{Z} \times \lambda_2 \mathbb{Z}, \qquad 
\Lambda^* = \kappa_1 \mathbb{Z} \times \kappa_2 \mathbb{Z}, \qquad 
\lambda_i \kappa_i = 2\pi,
\]
denote a dual pair of lattices. A scalar field \(f\) is called \((+)\)-symmetric if 
\(f \circ R_j = f\) for \(j = 1,2\), and a vector field \(\mathbf{f}\) is called 
\((+)\)-symmetric if \(\mathbf{f} \circ R_j = (-1)^j R_j \mathbf{f}\), where \(R_j\) 
denotes reflection of the \(j\)-th horizontal coordinate. We denote by \(C^2_+\) the 
space of \(C^2\) functions that are \(\Lambda\)-periodic and \((+)\)-symmetric.
The main theorem of this paper is as follows.
%We first state our main theorem, whose proof is given below.

\begin{theorem}\label{nn}
	Let \(U \in C^2([-d,0])\) be a non-constant function satisfying \(U(z) \neq 0\) for all \(z \in [-d,0]\). 
	Suppose that \((\tilde{\mathbf{u}}, \tilde{p}, \tilde{\eta})\) is a \(C^2\) curve of 
	\(C^2_+(\overline{\Omega^\eta}; \mathbb{R}^3) \times C^2_+(\overline{\Omega^\eta}) \times C^4_+(\mathbb{R}^2)\) 
	solutions to the steady hydroelastic wave equations \eqref{eq:momentum}--\eqref{eq:dynamic}, 
	parametrized by a small amplitude parameter \(\varepsilon \geq 0\), with bending energy \(W_b\in C^3\) satisfying
	\[
	W_b(0)=W_b'(0)=0,
	\qquad
	W_b''(0)>0.
	\]
	Assume that at \(\varepsilon = 0\), the solution coincides with the trivial shear flow
	\[
	\tilde{\mathbf{u}}|_{\varepsilon=0} = (U(z), 0, 0), \quad
	\tilde{p}|_{\varepsilon=0} = -g z, \quad
	\tilde{\eta}|_{\varepsilon=0} = 0.
	\]
Let \(\Phi\) be the flattening map defined by \eqref{map}, with Jacobian matrix
\(J=\nabla\Phi\) and
\[
\rho=\det J=1+\frac{\eta}{d}.
\]
Define the flattened surface elevation, modified pressure, and velocity on the fixed domain
\(\overline{\Omega}=\mathbb{R}^2\times[-d,0]\) by
\begin{align*}
	\eta(x,y)&=\tilde{\eta}(x,y),\\
	\mathscr{P}(x,y,z)
	&=\tilde p\!\left(x,y,z+\left(1+\frac{z}{d}\right)\eta(x,y)\right)
	+g\left(z+\left(1+\frac{z}{d}\right)\eta(x,y)\right),\\
	\mathbf u(x,y,z)
	&=\rho\,J^{-1}\,
	\tilde{\mathbf u}\!\left(x,y,z+\left(1+\frac{z}{d}\right)\eta(x,y)\right).
\end{align*}
	Let the asymptotic expansions of the flattened variables in powers of \(\varepsilon\):
	\begin{align*}
		\eta(x, y) &= \varepsilon \eta_1(x, y) + \varepsilon^2 \eta_2(x, y) + o(\varepsilon^2), \\
		\mathscr{P}(x, y, z) &= \varepsilon \mathscr{P}_1(x, y, z) + \varepsilon^2 \mathscr{P}_2(x, y, z) + o(\varepsilon^2), \\
		\mathbf{u}(x, y, z) &= \mathbf{u}_0(z) + \varepsilon \mathbf{u}_1(x, y, z) + \varepsilon^2 \mathbf{u}_2(x, y, z) + o(\varepsilon^2),
	\end{align*}
	with  \(\mathbf{u}_0(z) = (U(z), 0, 0)\),
then the first-order corrections satisfy
\[
v_1\equiv0,
\qquad
\partial_y\eta_1=0,
\qquad
\partial_y\mathscr P_1=0,
\qquad
\partial_y w_1=0.
\]
Moreover, if the first-order \(x\)-velocity perturbation has no
\(x\)-independent mode, in the sense that
\[
\int_0^{\lambda_1}u_1(x,y,z)\,dx=0,
\qquad (y,z)\in\mathbb R\times[-d,0],
\]
then
\[
\partial_yu_1
=
\partial_yv_1
=
\partial_yw_1
=
0.
\] In particular, under this normalization, the whole leading-order perturbation
is two-dimensional, in the sense that it is independent of the transverse
coordinate \(y\).
\end{theorem}

\begin{remark}\normalfont
	Although Theorem~\ref{nn} is stated for shear flows of the form
	\(\mathbf u_0=(U(z),0,0)\), we notice that 
		\begin{itemize}
		\item If \(U\) is non-constant, then the background flow has non-zero
		vorticity, namely
		\[
		\nabla\times\mathbf u_0=(0,U'(z),0).
		\]
		Thus the theorem applies to rotational shear flows, rather than only to
		perturbations of a uniform or irrotational background.
		
		\item The quadratic obstruction used below is closely related to the one
		developed for capillary-gravity waves in \cite{SVVW2026}, but the
		hydroelastic problem has a different linear boundary condition. The
		second-order capillary term is replaced by the fourth-order bending term
		from the hyperelastic shell model \cite{PT2011}. Consequently, the
		dispersion relation contains the hydroelastic factor \(|\mathbf k|^4\)
		instead of the capillary-gravity factor \(|\mathbf k|^2\). We therefore
		derive the linear kernel and the resonant set directly for the
		hydroelastic problem.
		
		\item After the linear kernel has been identified, the second-order obstruction
		comes from the interior Euler equations. This separates the part of the
		argument that depends on the hydroelastic boundary condition from the
		quadratic cancellation used to rule out transverse leading-order modes.
	\end{itemize}
\end{remark}

Hydroelastic wave models arise in the study of water waves beneath floating
elastic sheets, ice covers, and large flexible structures. They have been
used, for example, to describe waves generated by moving loads on ice
\cite{r3,r4,r5}, as well as nonlinear waves under elastic sheets and floating
ice plates \cite{PT2011,VandenBroeck2011,Guyenne2014}. From this point of
view, understanding whether small-amplitude hydroelastic waves remain
two-dimensional at leading order is relevant both to local bifurcation theory
and to the modelling of waves beneath floating elastic sheets in shear flows.

Despite extensive progress in the two-dimensional theory, genuinely
three-dimensional water waves remain much less understood. Known results
include explicit Gerstner-type equatorial solutions \cite{r7,r8}, edge waves
\cite{r9}, and several rigidity or non-existence results in the presence of
constant vorticity \cite{r10,r11,r12}. One of the main obstacles in the
three-dimensional theory is the appearance of resonances at quadratic order.
A formal construction of three-dimensional doubly periodic solutions with a
fixed boundary was given in \cite{r6}, although convergence of the resulting
series remains open. In the irrotational setting, three-dimensional doubly
periodic waves were first constructed by Reeder and Shinbrot \cite{r13};
subsequent work of Iooss and Plotnikov \cite{r14,r15} treated pure gravity
waves on infinite depth. More recently, Seth, Varholm, and Wahl\'en
\cite{SVW2024} proved the existence of symmetric doubly periodic
capillary-gravity waves with small vorticity bifurcating from uniform flows,
and Seth \cite{r1} obtained related existence results for internal waves. In
contrast, Seth et al.~\cite{SVVW2026} showed that symmetric
capillary-gravity waves bifurcating from non-uniform shear flows are
necessarily two-dimensional to leading order. The present paper establishes
the corresponding rigidity phenomenon for hydroelastic waves and shows that
the same quadratic obstruction remains effective after the capillary surface
energy is replaced by a fourth-order elastic bending energy.

In the capillary-gravity waves, the linear surface contribution
is a second-order surface-tension operator. In the hydroelastic waves, it is
replaced by a fourth-order bending operator. This changes the linear
dispersion relation: the hydroelastic waves contain a term proportional to
\(|\mathbf k|^4\), which gives an algebraic bound on the admissible wave
numbers. More precisely, the dynamic condition~\eqref{eq:dynamic}--\eqref{eq:Psi_n}
involves the mean curvature \(H\), the Gaussian curvature \(K\), and the
Laplace--Beltrami operator \(\Delta_\Sigma\) on the deformed surface. A
small-amplitude expansion gives
\[
H
=
\frac12\Delta\eta
+
O\!\left(|\nabla\eta|^2|D^2\eta|\right),
\qquad
K
=
O\!\left(|D^2\eta|^2\right),
\qquad
\Delta_\Sigma
=
\Delta
+
O\!\left(|\nabla\eta|^2D^2\right),
\]
where \(D=(\partial_x,\partial_y)\) denotes differentiation in the horizontal
variables. The precise linearization is carried out in
Section~\ref{sect3.3}. At leading order, the normal elastic stress is
therefore
\[
\frac{W_b''(0)}4\Delta^2\eta.
\]
Consequently, the linearized dynamic boundary condition becomes
\[
\mathscr P_1
=
g\eta_1
-
\frac{W_b''(0)}4\Delta^2\eta_1
\qquad\text{on } z=0,
\]
and the dispersion relation~\eqref{eq:dispersion} takes the form
\[
q_{\mathbf k}(0)
=
\frac{\alpha^2U(0)^2}
{g-\frac{W_b''(0)}4|\mathbf k|^4}.
\]
Since \(q_{\mathbf k}(0)>0\), every resonant mode must satisfy
\[
g-\frac{W_b''(0)}4|\mathbf k|^4>0,
\]
and hence
\[
|\mathbf k|^4<\frac{4g}{W_b''(0)}.
\]
Thus the hydroelastic bending stiffness gives a direct algebraic truncation
of the resonant set \(\Sigma(U)\); see Proposition~\ref{prop:finite_modes}.
This is different from the capillary-gravity case, where finiteness is
obtained through the large-\(|\mathbf k|\) asymptotic estimate
\[
\frac{q_{\mathbf k}(0)}{|\mathbf k|}
\to 1,
\]
proved by comparison for the Riccati equation~\eqref{eq:riccati}. Thus,
although the same Riccati equation governs the vertical pressure profile in
both settings, the hydroelastic bending energy changes the spectral
selection mechanism at the free surface.

The proof of Theorem~\ref{nn} proceeds in three steps. First, we flatten the
free-boundary problem and linearize about a non-uniform non-stagnant shear
flow. The fourth-order bending term enters through the linear dynamic
boundary condition and produces an algebraic bound on the resonant Fourier
modes. Second, at quadratic order, we derive a necessary solvability
condition by taking the first component of the curl of the momentum equation
and averaging over one period in the \(x\)-direction. This condition comes from the
Euler equations in the fluid interior and is independent of the explicit
second-order expansion of the hydroelastic boundary operator. Third,
inserting the reduced kernel representation into this interior solvability
condition and using the Riccati equation for the pressure shows that every
Fourier mode with transverse dependence must vanish. It follows that the
first-order surface elevation, pressure, transverse velocity component, and
vertical velocity component are independent of the transverse coordinate, up
to the shear-direction zero-mode in the first horizontal velocity component.

\subsection{Outline}
The outline of this paper is organized as follows. Section~\ref{sect3} introduces the flattening transformation and the symmetry class under consideration. In Section~\ref{sect4}, we linearize the problem about a non-uniform shear flow and characterize the kernel of the linearized operator. Section~\ref{sect5} derives the second-order solvability condition. We prove Theorem~\ref{nn} in Section~\ref{sect6}.
	\section{Flattening and linearization}\label{sect3}
	In this section, we transform the free-boundary hydroelastic problem to a
	fixed fluid domain and compute the linearized equations around the
	non-uniform shear flow.
	\subsection{Flattening transform}
	To work on a fixed domain, we introduce a flattening transformation that maps the unknown free surface \(z = \eta(x, y)\) to a fixed boundary. Following the approach in \cite{SVW2024,SVVW2026,r1}, we define the map
	\begin{align}
		\Phi: \overline{\Omega} &\to \overline{\Omega^\eta}, \notag\nonumber\\
		(\bar{x}, \bar{y}, \bar{z}) &\mapsto (x, y, z) = (\bar{x}, \bar{y}, \bar{z} + \varphi(\bar{x}, \bar{y}, \bar{z}))\label{map},
	\end{align}
	where \(\Omega = \mathbb{R}^2 \times (-d, 0)\) is the fixed flat domain. The function \(\varphi\) is chosen to straighten the free surface:
	\begin{align*}
		\varphi(\bar{x}, \bar{y}, \bar{z}) = \left(1 + \frac{\bar{z}}{d}\right) \eta(\bar{x}, \bar{y}).
	\end{align*}
	This linear interpolation ensures that \(\Phi\) maps the fixed boundaries \(z = -d\) and \(z = 0\) onto the physical bottom \(z = -d\) and the free surface \(z = \eta(x, y)\), respectively.
	The Jacobian matrix of \(\Phi\) is given by
	\begin{align*}
		J = \nabla \Phi = I + \mathbf{e}_3 \otimes \nabla \varphi,
	\end{align*}
	where \(\nabla\) denotes the gradient with respect to \((\bar{x}, \bar{y}, \bar{z})\). Its determinant is
	\begin{align*}
		\det J = 1 + \partial_{\bar{z}} \varphi = 1 + \frac{\eta}{d} =: \rho.
	\end{align*}
For a vector field \(\mathbf h\) on \(\Omega^\eta\), we define its flattened counterpart
\(\bar{\mathbf h}\) on \(\Omega\) by the Piola transformation
\begin{align*}
	J\bar{\mathbf h}=\rho\,\mathbf h\circ\Phi,
	\qquad\text{equivalently}\qquad
	\bar{\mathbf h}=\rho J^{-1}(\mathbf h\circ\Phi).
\end{align*}
This transformation preserves the divergence-free constraint in the sense that
\begin{align}
	\nabla\cdot\bar{\mathbf h}
	=
	\rho\,(\nabla\cdot\mathbf h)\circ\Phi .
	\label{eq:Piola-div}
\end{align}
In particular, if \(\nabla\cdot\mathbf h=0\) in the physical domain, then
\(\nabla\cdot\bar{\mathbf h}=0\) in the fixed domain.

We now apply this transformation to the velocity field, thus the flattened velocity
\(\bar{\mathbf u}\) is defined by
\[
\bar{\mathbf u}=\rho J^{-1}(\mathbf u\circ\Phi).
\]
which implies
\[
\mathbf u\circ\Phi=\rho^{-1}J\bar{\mathbf u}.
\]
For convenience, set
\[
\mathcal M:=\rho^{-1}J.
\]
Thus \(\mathbf u\circ\Phi=\mathcal M\bar{\mathbf u}\).

The pressure is a scalar and therefore is not transformed by the Piola map. Instead, we first
introduce the modified physical pressure
\[
\mathscr P^{\eta}:=p+gz
\]
in \(\Omega^\eta\), and then define its flattened counterpart by the scalar pullback
\[
\bar{\mathscr P}:=\mathscr P^{\eta}\circ\Phi .
\]
Consequently, one has
\[
(\nabla \mathscr P^{\eta})\circ\Phi
=
J^{-\top}\nabla\bar{\mathscr P}.
\]

The steady Euler equations in the moving frame are
\begin{align*}
	(\mathbf u\cdot\nabla)\mathbf u+\nabla\mathscr P^{\eta}&=0,\\
	\nabla\cdot\mathbf u&=0.
\end{align*}
Pulling back the momentum equation by \(\Phi\) gives
\[
\frac1{\rho}(\bar{\mathbf u}\cdot\nabla)(\mathcal M\bar{\mathbf u})
+
J^{-\top}\nabla\bar{\mathscr P}
=0.
\]
Multiplying by \(J^\top\) and using \(\mathcal M^\top=\rho^{-1}J^\top\), we obtain the fixed-domain
momentum equation
\[
\mathcal M^\top(\bar{\mathbf u}\cdot\nabla)(\mathcal M\bar{\mathbf u})
+
\nabla\bar{\mathscr P}
=0.
\]
Together with \eqref{eq:Piola-div}, the transformed bulk equations are 
\begin{align*}
	\mathcal M^\top(\bar{\mathbf u}\cdot\nabla)(\mathcal M\bar{\mathbf u})
	+\nabla\bar{\mathscr P}&=0,\\
	\nabla\cdot\bar{\mathbf u}&=0.
\end{align*}
The boundary conditions also transform. The bottom impermeability condition \eqref{eq:bottom} becomes
	\begin{align*}
		\bar{w} = 0 \quad \text{on } z = -d. 
	\end{align*}
The kinematic condition \eqref{eq:kinematic} transforms, under the Piola
map, into a no-flux condition on the fixed upper boundary:
\begin{align*}
	\bar{w}=0 \quad \text{on } z=0.
\end{align*}
Indeed, the Piola transformation preserves normal fluxes. Thus, the
physical condition
\[
\tilde{\mathbf u}\cdot(-\eta_x,-\eta_y,1)=0
\quad \text{on } \tilde z=\eta(x,y)
\]
is transformed into the simple flattened condition \(\bar w=0\) on the flat
boundary \(z=0\).
	After dropping the bars, Eqs. \eqref{eq:momentum}-\eqref{eq:dynamic} are equivalent to
	\begin{align}
		\mathcal{M}^\top (\mathbf{u} \cdot \nabla)(\mathcal{M} \mathbf{u}) + \nabla \mathscr{P} &= 0, \label{eq:flattened_momentum}\\
		\nabla \cdot \mathbf{u} &= 0, \label{eq:flattened_continuity}\\
		w &= 0 \quad \text{on } z = -d, \label{eq:flattened_bottom}\\
		w &= 0\quad \text{on } z = 0, \label{eq:flattened_kinematic}\\
		\mathscr{P} - g\eta + \dfrac{1}{\sqrt{a}}\Psi_n &= 0 \quad \text{on } z = 0, \label{eq:flattened_dynamic}
	\end{align}
	where \(\mathcal{M} = \rho^{-1} J\) and \(\Psi_n\) is given by \eqref{eq:Psi_n}.
	
	\subsection{Trivial solutions and symmetries}
	A family of trivial solutions is given by shear flows with a flat surface:
	\begin{align*}
		\mathbf{u}_0 = (U(z), 0, 0), \quad p_0(z) = -g z, \quad \eta_0 \equiv 0, 
	\end{align*}
	where \(U \in C^2([-d,0])\) is arbitrary. Such a flow is called \textit{non-uniform} if \(U\) is not constant, and \textit{non-stagnant} if \(U\) has no zeros on \([-d,0]\).
	
We restrict to functions which are doubly periodic with respect to a lattice
\begin{align*}
	\Lambda = \lambda_1 \mathbb{Z} \times \lambda_2 \mathbb{Z}, \quad \Lambda^* = \kappa_1 \mathbb{Z} \times \kappa_2 \mathbb{Z}, \quad \lambda_i \cdot \kappa_j = 2\pi \delta_{ij},
\end{align*}
where \(\delta_{ij}\) denotes the Kronecker delta. The functions satisfy the \((+)\)-symmetry condition: for reflections \(R_x\) in \(x\) and \(R_y\) in \(y\),
\begin{align*}
	\mathbf{u} \circ R_x = -R_x \mathbf{u}, \quad \mathbf{u} \circ R_y = R_y \mathbf{u}, \quad f \circ R_x = f, \quad f \circ R_y = f \quad \text{for scalar fields } f = \mathscr{P}, \eta,
\end{align*}
where \(R_x\) and \(R_y\) act on vectors as
\begin{align*}
	R_x (u, v, w) = (-u, v, w), \qquad R_y (u, v, w) = (u, -v, w).
\end{align*}
This symmetry class is preserved by the equations. In Fourier series, this means
\begin{align}
	\mathbf{u}(x,y,z) &= \sum_{\mathbf{k} \in \Lambda^*} \begin{pmatrix}
		u_{\mathbf{k}}(z) \cos(\alpha x) \cos(\beta y) \\
		- v_{\mathbf{k}}(z) \sin(\alpha x) \sin(\beta y) \\
		i w_{\mathbf{k}}(z) \sin(\alpha x) \cos(\beta y)
	\end{pmatrix},\\
	\eta(x,y) &= \sum_{\mathbf{k} \in \Lambda^*} \eta_{\mathbf{k}} \cos(\alpha x) \cos(\beta y),\\
	\mathscr{P}(x,y,z) &= \sum_{\mathbf{k} \in \Lambda^*} \mathscr{P}_{\mathbf{k}}(z) \cos(\alpha x) \cos(\beta y), \label{eq:Fourier_p}
\end{align}
where $\mathbf{k} = (\alpha, \beta) \in \Lambda^*$.
	
	\subsection{Linearization}\label{sect3.3}

We expand all quantities in powers of \(\varepsilon\). Let
\begin{align*}
	\mathbf{u} = \mathbf{u}_0 + \varepsilon \mathbf{u}_1 + \varepsilon^2 \mathbf{u}_2 + o(\varepsilon^2),
\end{align*}
and similarly for \(\mathscr{P}\) and \(\eta\). For the flattened pressure, we have \(\mathscr{P} = \varepsilon \mathscr{P}_1 + \varepsilon^2 \mathscr{P}_2 + o(\varepsilon^2)\) since \(\mathscr{P}_0 = 0\) in the background state. The surface elevation expands as \(\eta = \varepsilon \eta_1 + \varepsilon^2 \eta_2 + o(\varepsilon^2)\). The matrix \(\mathcal{M}\) expands as
\begin{align*}
	\mathcal{M} = I + \varepsilon \mathcal{M}_1 + \varepsilon^2 \mathcal{M}_2 + o(\varepsilon^2),
\end{align*}
with \(\mathcal{M}_1 = \mathbf{e}_3 \otimes \nabla \varphi_1 - \partial_z \varphi_1 I\)  and \(\varphi_1 = (1 + z/d) \eta_1\).

At order \(\varepsilon\), equation \eqref{eq:flattened_momentum} gives
\begin{align}
	U \partial_x \mathbf{u}_1 + w_1 U' \mathbf{e}_1 + U^2 \partial_x \nabla \times (\varphi_1 \mathbf{e}_2) + \nabla \mathscr{P}_1 = 0. \label{eq:linear_momentum}
\end{align}
Incompressibility \eqref{eq:flattened_continuity} yields
\begin{align}\label{incompress}
	\nabla \cdot \mathbf{u}_1 = 0.
\end{align}
The flattened impermeability conditions \eqref{eq:flattened_bottom} and
\eqref{eq:flattened_kinematic} give, at first order,
\begin{align}
	w_1 &=0 \quad \text{on } z=-d, \label{eq:linear_bottom}\\
	w_1 &=0 \quad \text{on } z=0.
\end{align}

To linearize the dynamic boundary condition, we expand the normal stress density 
\(\Psi_n\) given by \eqref{eq:Psi_n} in powers of \(\varepsilon\). Substituting 
\(\eta = \varepsilon \eta_1 + \varepsilon^2 \eta_2 + o(\varepsilon^2)\) into the geometric 
quantities \eqref{eq:H}--\eqref{eq:LaplaceBeltrami} yields
\begin{align}
	H
	&=
	\frac{\varepsilon}{2}\Delta\eta_1
	+
	\frac{\varepsilon^2}{2}\Delta\eta_2
	+
	o(\varepsilon^2),
	\label{eq:H_expansion}\\
	K
	&=
	\varepsilon^2
	\bigl(
	\partial_{xx}\eta_1\partial_{yy}\eta_1
	-
	(\partial_{xy}\eta_1)^2
	\bigr)
	+
	o(\varepsilon^2),
	\label{eq:K_expansion}\\
	\Delta_\Sigma
	&=
	\Delta
	+
	O(\varepsilon^2),
	\label{eq:LB_expansion}\\
	\sqrt a
	&=
	1
	+
	\frac{\varepsilon^2}{2}|\nabla\eta_1|^2
	+
	o(\varepsilon^2).
	\label{eq:area_expansion}.
\end{align}
 Assuming the shell is unstressed in the flat state, we have 
\(W_b'(0) = 0\) and \(W_b(0) = 0\). Expanding the bending energy about \(H = 0\) gives
\begin{align}
	W_b(H) &= \frac{1}{2} W_b''(0) H^2 + \frac{1}{6} W_b'''(0) H^3 + o(H^3), \label{eq:Wb_expansion} \\
	W_b'(H) &= W_b''(0) H + \frac{1}{2} W_b'''(0) H^2 + o(H^2). \label{eq:Wbp_expansion}
\end{align}
Inserting \eqref{eq:H_expansion} into \eqref{eq:Wbp_expansion} yields
\begin{align}
	W_b'(H) = \frac{\varepsilon}{2} W_b''(0) \Delta \eta_1 
	+ \frac{\varepsilon^2}{2} W_b''(0) \Delta \eta_2 
	+ \frac{\varepsilon^2}{8} W_b'''(0) (\Delta \eta_1)^2 
	+ o(\varepsilon^2). \label{eq:Wbp_expanded}
\end{align}
We now evaluate each term in \eqref{eq:Psi_n} up to order \(\varepsilon\).
	Using  \eqref{eq:LB_expansion}-\eqref{eq:area_expansion}and \eqref{eq:Wbp_expanded}, we have
\begin{align*}
	\frac{\sqrt{a}}{2} \Delta_\Sigma W_b'(H) 
	&= 
	\frac{1}{2}
	\Bigl(1 + O(\varepsilon^2)\Bigr)
	\bigl(\Delta + O(\varepsilon^2)\bigr)
	\Bigl(
	\frac{\varepsilon}{2} W_b''(0) \Delta \eta_1
	+
	O(\varepsilon^2)
	\Bigr) \\
	&=
	\frac{\varepsilon}{4}W_b''(0)\Delta^2\eta_1
	+
	O(\varepsilon^2).
\end{align*}
	From \eqref{eq:Wb_expansion} and \eqref{eq:Wbp_expansion}, we obtain
	\begin{align*}
		H W_b'(H) - W_b(H) 
		&= H \bigl( W_b''(0) H + \tfrac{1}{2} W_b'''(0) H^2 + o(H^2) \bigr)\\[4pt]
		&\quad
		- \bigl( \tfrac{1}{2} W_b''(0) H^2 + \tfrac{1}{6} W_b'''(0) H^3 + o(H^3) \bigr) \\[4pt]
		&= \frac{1}{2} W_b''(0) H^2 + \frac{1}{3} W_b'''(0) H^3 + o(H^3).
	\end{align*}
	Since \(H = O(\varepsilon)\) by \eqref{eq:H_expansion}, we have 
	\(H W_b'(H) - W_b(H) = O(\varepsilon^2)\). With 
	\(\sqrt{a} = 1 + O(\varepsilon^2)\) and \(H = O(\varepsilon)\), then 
	\[2\sqrt{a}H(H W_b'(H) - W_b(H))=
	O(\varepsilon^3).\]
	By \eqref{eq:K_expansion}, \(K = O(\varepsilon^2)\), and by \eqref{eq:Wbp_expanded}, 
	\(W_b'(H) = O(\varepsilon)\). Hence, we have
	\[ -\sqrt{a} K W_b'(H)=O(\varepsilon^3).\]
Combined with  the above, the normal stress density expands as
\begin{align}
	\Psi_n = \varepsilon \frac{W_b''(0)}{4} \Delta^2 \eta_1 + \varepsilon^2 \Psi_{n,2} + o(\varepsilon^2), \label{eq:Psi_n_expanded}
\end{align}
where \(\Psi_{n,2}\) collects all terms at order \(\varepsilon^2\) from the bending energy 
expansion. Its explicit form is not required for the linearized analysis.
The flattened dynamic boundary condition is
\[
\mathscr P-g\eta+\frac1{\sqrt a}\Psi_n=0
\qquad\text{on }~z=0.
\]
Since \(1/\sqrt a=1+O(\varepsilon^2)\), this factor does not affect the
terms at order \(\varepsilon\). Substituting the expansions of
\(\mathscr P\), \(\eta\), and \(\Psi_n\), the terms at order
\(\varepsilon\) give
\[
\mathscr P_1
=
g\eta_1
-
\frac{W_b''(0)}4\Delta^2\eta_1
\qquad\text{on }~z=0. \label{eq:linear_dynamic}
\]

	\section{Analysis of the linearized problem}\label{sect4}
In this section, we first solve the linearized problem mode by mode, from which we derive the hydroelastic dispersion relation. Subsequently, we prove the finiteness of the resonant set, and finally characterize the kernel of the linearized operator.
	\subsection{Reduction to a scalar problem}
Taking the \(z\)-component of the linearized momentum equation
\eqref{eq:linear_momentum}, we obtain
\begin{align}
	U\partial_x w_1+\partial_z\mathscr P_1
	=
	-\,U^2\partial_x^2\varphi_1 .
	\label{eq:third_comp}
\end{align}
Indeed, one has
\[
\nabla\times(\varphi_1\mathbf e_2)
=
(-\partial_z\varphi_1,0,\partial_x\varphi_1),
\]
so that the third component of
\(U^2\partial_x\nabla\times(\varphi_1\mathbf e_2)\) is
\(U^2\partial_x^2\varphi_1\), which appears on the left-hand side of
\eqref{eq:linear_momentum}.
Taking the divergence of \eqref{eq:linear_momentum} and using \eqref{incompress} give
\begin{align}
	2U'\partial_x w_1+\Delta\mathscr P_1
	=
	-\,2UU'\partial_x^2\varphi_1 .
	\label{eq:div_mom}
\end{align}
From \eqref{eq:third_comp}, we have
\begin{align}\label{eq1}
	\partial_x w_1
	=
	-\,U\partial_x^2\varphi_1
	-
	\frac{\partial_z\mathscr P_1}{U}.
\end{align}
Substituting \eqref{eq1} into \eqref{eq:div_mom} yields
\begin{align*}
	2U'\left(
	-\,U\partial_x^2\varphi_1
	-
	\frac{\partial_z\mathscr P_1}{U}
	\right)
	+
	\Delta\mathscr P_1
	=
	-\,2UU'\partial_x^2\varphi_1 .
\end{align*}
The terms involving \(\partial_x^2\varphi_1\) cancel, and hence
\begin{align}
	\Delta \mathscr{P}_1
	-
	2\frac{U'}{U}\partial_z\mathscr{P}_1
	=0
	\quad \text{in } \Omega.
	\label{eq:wp_eq}
\end{align}
	
	\subsection{Fourier representation}
	We now express the linearized problem in Fourier series. Recall the Fourier expansions \eqref{eq:Fourier_p}:
	\begin{align*}
		\mathscr{P}_1(x, y, z) &= \sum_{\mathbf{k} \in \Lambda^*} \mathscr{P}_{1,\mathbf{k}}(z) \textup{e}^{i(\alpha x + \beta y)},
	\end{align*}
	where $\mathbf{k} = (\alpha, \beta) \in \Lambda^*$ and we work with complex exponentials for convenience; the final result for the cosine series follows by linearity.
	Substituting the Fourier expansion of $\mathscr{P}_1$ into \eqref{eq:wp_eq} yields for each mode $\mathbf{k}$
	\begin{align}
		\mathscr{P}_{1,\mathbf{k}}'' - 2\frac{U'}{U} \mathscr{P}_{1,\mathbf{k}}' - |\mathbf{k}|^2 \mathscr{P}_{1,\mathbf{k}} = 0, \quad z \in (-d,0). \label{eq:wp_ode}
	\end{align}
	This can be written in divergence form as
	\begin{align*}
		\left( \frac{\mathscr{P}_{1,\mathbf{k}}'}{U^2} \right)' = |\mathbf{k}|^2 \frac{\mathscr{P}_{1,\mathbf{k}}}{U^2}. 
	\end{align*}
	
From \eqref{eq:linear_bottom}, we have
\[
w_{1,\mathbf k}(-d)=0.
\]
Evaluating \eqref{eq:third_comp} at \(z=-d\) gives
\[
U(-d)i\alpha w_{1,\mathbf k}(-d)
+
\mathscr P_{1,\mathbf k}'(-d)
=
-\,U(-d)^2 \partial_x^2\varphi_{1,\mathbf k}(-d).
\]
Since
\[
\varphi_1=\left(1+\frac zd\right)\eta_1,
\]
we have \(\varphi_{1,\mathbf k}(-d)=0\). Hence
\begin{align*}
	\mathscr P_{1,\mathbf k}'(-d)=0.
\end{align*}

Next, the flattened kinematic condition on the free surface is
\[
w_1=0
\quad \text{on } z=0,
\]
and therefore, for each Fourier mode, we have
\[
w_{1,\mathbf k}(0)=0.
\]
Evaluating \eqref{eq:third_comp} at \(z=0\), we obtain
\[
U(0)i\alpha w_{1,\mathbf k}(0)
+
\mathscr P_{1,\mathbf k}'(0)
=
-\,U(0)^2 \partial_x^2\varphi_{1,\mathbf k}(0).
\]
Since \(w_{1,\mathbf k}(0)=0\) and \(\varphi_{1,\mathbf k}(0)=\eta_{1,\mathbf k}\), which yields
\begin{align*}
	\mathscr P_{1,\mathbf k}'(0)
	=
	\alpha^2 U(0)^2\eta_{1,\mathbf k}.
\end{align*}
	The linearized dynamic condition \eqref{eq:linear_dynamic} reads
	\begin{align*}
		\mathscr{P}_1 = g \eta_1 - \frac{W_b''(0)}{4} \Delta^2 \eta_1 \quad \text{on } z = 0.
	\end{align*}
	For a Fourier mode $\ep^{i(\alpha x + \beta y)}$, we have $\Delta^2 \eta_1 = |\mathbf{k}|^4 \eta_1$ with $|\mathbf{k}|^2 = \alpha^2 + \beta^2$. Hence, we obtain
	\begin{align*}
		\mathscr{P}_{1,\mathbf{k}}(0) = \left( g - \frac{W_b''(0)}{4} |\mathbf{k}|^4 \right) \eta_{1,\mathbf{k}}.
	\end{align*}
	Together with the dynamic boundary condition, the boundary conditions for
	\(\mathscr P_{1,\mathbf k}\) are
	\begin{align*}
		\mathscr P_{1,\mathbf k}'(-d)&=0,\\
		\mathscr P_{1,\mathbf k}'(0)&=\alpha^2U(0)^2\eta_{1,\mathbf k},\\
		\mathscr P_{1,\mathbf k}(0)
		&=
		\left(
		g-\frac{W_b''(0)}4|\mathbf k|^4
		\right)\eta_{1,\mathbf k}.
	\end{align*}
For each \(\mathbf k\neq \mathbf 0\), set
\[
\kappa:=|\mathbf k|>0.
\]
Since the pressure mode \(\mathscr P_{1,\mathbf k}\) satisfies the same
vertical equation and the same bottom condition up to a multiplicative
constant, it is useful to introduce the normalized profile \(P_{\mathbf k}\)
defined by
\begin{align}
	P_{\mathbf k}''
	-
	2\frac{U'}{U}P_{\mathbf k}'
	-
	\kappa^2P_{\mathbf k}
	&=0,
	\qquad -d<z<0,
	\label{eq:Pk_ode}\\
	P_{\mathbf k}(-d)&=1,
	\qquad
	P_{\mathbf k}'(-d)=0.
	\label{eq:Pk_initial}
\end{align}
Since \(U(z)\neq0\) on \([-d,0]\), equation \eqref{eq:Pk_ode} can be written
in divergence form as
\begin{align}
	\left(\frac{P_{\mathbf k}'}{U^2}\right)'
	=
	\kappa^2\frac{P_{\mathbf k}}{U^2}.
	\label{eq:Pk_divergence}
\end{align}

\begin{lemma}\label{lem:riccati}
	For each \(\mathbf k\neq\mathbf 0\), the solution \(P_{\mathbf k}\) of
	\eqref{eq:Pk_ode}--\eqref{eq:Pk_initial} satisfies
	\[
	P_{\mathbf k}(z)>0
	\quad\text{for all }z\in[-d,0],
	\]
	and
	\[
	P_{\mathbf k}'(z)>0
	\quad\text{for all }z\in(-d,0],
	\]
	then
	\[
	q_{\mathbf k}(z):=\frac{P_{\mathbf k}'(z)}{P_{\mathbf k}(z)}
	\]
	is well-defined on \([-d,0]\), belongs to \(C^1([-d,0])\), and satisfies
	\begin{align}
		q_{\mathbf k}'
		=
		2\frac{U'}{U}q_{\mathbf k}+
		\kappa^2-q_{\mathbf k}^2,
		\qquad
		q_{\mathbf k}(-d)=0.
		\label{eq:riccati}
	\end{align}
	Moreover, we have
	\begin{align*}
		q_{\mathbf k}(0)>0,
		\qquad
		\lim_{\kappa\to\infty}
		\frac{q_{\mathbf k}(0)}{\kappa}
		=1.
	\end{align*}
\end{lemma}

\begin{proof}
	The existence and uniqueness of \(P_{\mathbf k}\in C^2([-d,0])\) follow
	from the standard theory of linear ordinary differential equations, since
	\(U'/U\in C^1([-d,0])\).
	
	We first prove $P_{\mathbf k}(z)>0$ on $[-d,0]$. Assume that
	\(P_{\mathbf k}\) has a zero in \((-d,0]\). Let \(z_0\) be its first zero, 
	then
	\[
	P_{\mathbf k}(z)>0
	\quad\text{for } -d\le z<z_0,
	\qquad
	P_{\mathbf k}(z_0)=0.
	\]
	On \((-d,z_0)\),  \eqref{eq:Pk_divergence} gives
	\[
	\left(\frac{P_{\mathbf k}'}{U^2}\right)'
	=
	\kappa^2\frac{P_{\mathbf k}}{U^2}
	>0.
	\]
	Since \(P_{\mathbf k}'(-d)=0\), it follows that
	\[
	\frac{P_{\mathbf k}'(z)}{U(z)^2}>0
	\quad\text{for } -d<z<z_0.
	\]
	Hence \(P_{\mathbf k}'(z)>0\) on \((-d,z_0)\). Therefore, one has
	\[
	P_{\mathbf k}(z_0)>P_{\mathbf k}(-d)=1,
	\]
	which contradicts \(P_{\mathbf k}(z_0)=0\). Thus, it holds that
	\[
	P_{\mathbf k}(z)>0
	\quad\text{on }~[-d,0].
	\]
	
	Returning to \eqref{eq:Pk_divergence}, we obtain
	\[
	\left(\frac{P_{\mathbf k}'}{U^2}\right)'
	=
	\kappa^2\frac{P_{\mathbf k}}{U^2}
	>0
	\quad\text{on }~[-d,0].
	\]
	Together with \(P_{\mathbf k}'(-d)=0\), this yields
	\[
	P_{\mathbf k}'(z)>0
	\quad\text{for } -d<z\le0.
	\]
	In particular, \(q_{\mathbf k}\) is globally well-defined on \([-d,0]\),
	and
	\[
	q_{\mathbf k}(0)>0.
	\]
	
	Next, differentiating
	\[
	q_{\mathbf k}=\frac{P_{\mathbf k}'}{P_{\mathbf k}}
	\]
	and combining with \eqref{eq:Pk_ode}, we obtain
	\[
	q_{\mathbf k}'
	=
	\frac{P_{\mathbf k}''}{P_{\mathbf k}}
	-
	\left(
	\frac{P_{\mathbf k}'}{P_{\mathbf k}}
	\right)^2
	=
	2\frac{U'}{U}q_{\mathbf k}
	+
	\kappa^2
	-
	q_{\mathbf k}^2
	\]
	and
	\[
	q_{\mathbf k}(-d)
	=
	\frac{P_{\mathbf k}'(-d)}{P_{\mathbf k}(-d)}
	=
	0,
	\]
	which proves \eqref{eq:riccati}.
	
	It remains to prove the asymptotic formula. Set
	\[
	a(z):=\frac{U'(z)}{U(z)},
	\qquad
	A:=\|a\|_{L^\infty(-d,0)}.
	\]
	Since \(q_{\mathbf k}\ge0\), equation \eqref{eq:riccati} implies
	\[
	\kappa^2-q_{\mathbf k}^2-2Aq_{\mathbf k}
	\le
	q_{\mathbf k}'
	\le
	\kappa^2-q_{\mathbf k}^2+2Aq_{\mathbf k}.
	\]
	Let \(\ell_-\) and \(\ell_+\) solve
	\begin{align*}
		\ell_-'
		&=
		\kappa^2-\ell_-^2-2A\ell_-,
		\qquad
		\ell_-(-d)=0,\\
		\ell_+'
		&=
		\kappa^2-\ell_+^2+2A\ell_+,
		\qquad
		\ell_+(-d)=0.
	\end{align*}
	By the scalar comparison principle for ordinary differential equations,
	\[
	\ell_-(z)
	\le
	q_{\mathbf k}(z)
	\le
	\ell_+(z),
	\qquad -d\le z\le0.
	\]
	We compute the two comparison functions explicitly. For \(c\in\{-A,A\}\),
	let \(\ell_c\) solve
	\[
	\ell_c'
	=
	\kappa^2-\ell_c^2+2c\ell_c,
	\qquad
	\ell_c(-d)=0, 
	\]
we obtain	
	\[
	\ell_c(z)
	=
	c
	+
	\mu_c
	\frac{
		R_c \textup{e}^{2\mu_c(z+d)}-1
	}{
		R_c \textup{e}^{2\mu_c(z+d)}+1
	},
	\]	
where
\[
R_c:=\frac{\mu_c-c}{\mu_c+c},~~~~	\mu_c:=\sqrt{\kappa^2+c^2}.
	\]	
	Hence, at \(z=0\), one has
	\[
	\ell_c(0)
	=
	c
	+
	\mu_c
	\frac{
		R_c \textup{e}^{2\mu_c d}-1
	}{
		R_c \textup{e}^{2\mu_c d}+1
	}.
	\]
	When \(\kappa\to\infty\), we have
	\[
	\mu_c=\kappa+O(\kappa^{-1}),
	\qquad
	R_c=1+O(\kappa^{-1}),
	\qquad
	\textup{e}^{-2\mu_c d}=o(1).
	\]
	Therefore
	\[
	\ell_c(0)=\kappa+O(1),
	\qquad
	\frac{\ell_c(0)}{\kappa}\to1.
	\]
	Using
	\[
	\ell_-(0)
	\le
	q_{\mathbf k}(0)
	\le
	\ell_+(0),
	\]
	we conclude that
	\[
	\lim_{\kappa\to\infty}
	\frac{q_{\mathbf k}(0)}{\kappa}
	=1.
	\]
	The proof is completed.

\end{proof}

\subsection{Dispersion relation}

Let \(\mathbf k\neq\mathbf0\) be a nontrivial Fourier mode. Since the
solution space of \eqref{eq:wp_ode} satisfying
\[
\mathscr P_{1,\mathbf k}'(-d)=0
\]
is one-dimensional, every nontrivial solution is a scalar multiple of the
normalized solution \(P_{\mathbf k}\) constructed in Lemma
\ref{lem:riccati}. Therefore, we have
\[
\frac{\mathscr P_{1,\mathbf k}'(0)}
{\mathscr P_{1,\mathbf k}(0)}
=
\frac{P_{\mathbf k}'(0)}{P_{\mathbf k}(0)}
=
q_{\mathbf k}(0),
\]
provided \(\mathscr P_{1,\mathbf k}\not\equiv0\).

Using the boundary conditions
\[
\mathscr P_{1,\mathbf k}'(0)
=
\alpha^2U(0)^2\eta_{1,\mathbf k},
\]
and
\[
\mathscr P_{1,\mathbf k}(0)
=
\left(
g-\frac{W_b''(0)}4|\mathbf k|^4
\right)\eta_{1,\mathbf k},
\]
we obtain the dispersion relation
\begin{align}
	q_{\mathbf k}(0)
	=
	\frac{\alpha^2U(0)^2}
	{
		g-\frac{W_b''(0)}4|\mathbf k|^4
	}.
	\label{eq:dispersion}
\end{align}
Notice that Lemma \ref{lem:riccati} implies
\[
q_{\mathbf k}(0)>0
\quad\text{for every } \mathbf k\neq\mathbf0.
\]

\begin{proposition}\label{prop:finite_modes}
	The set of resonant wave vectors
	\[
	\Sigma(U)
	:=
	\{
		\mathbf k\in\Lambda^*\setminus\{\mathbf0\}:
		\mathbf k \text{ satisfies } \eqref{eq:dispersion}
		\}
	\]
	is finite.
\end{proposition}

\begin{proof}
	By Lemma \ref{lem:riccati}, we have
	\[
	q_{\mathbf k}(0)>0
	\quad
	\text{for every } \mathbf k\neq\mathbf0.
	\]
	If \(\mathbf k\in\Sigma(U)\), then the dispersion relation
	\eqref{eq:dispersion} holds. Hence the right-hand side of
	\eqref{eq:dispersion} must be positive:
	\[
	\frac{\alpha^2U(0)^2}
	{
		g-\frac{W_b''(0)}4|\mathbf k|^4
	}
	>0.
	\]
	Since the flow is non-stagnant, \(U(0)\neq0\). Moreover, the case
	\(\alpha=0\) is impossible, because then the right-hand side of
	\eqref{eq:dispersion} would be zero, contradicting
	\(q_{\mathbf k}(0)>0\). Therefore \(\alpha^2U(0)^2>0\), and we have
	\[
	g-\frac{W_b''(0)}4|\mathbf k|^4>0.
	\]
	which implies
	\[
	|\mathbf k|^4
	<
	\frac{4g}{W_b''(0)}.
	\]
	Consequently, every \(\mathbf k\in\Sigma(U)\) lies in the bounded lattice
	set
	\[
	\left\{\mathbf k\in\Lambda^*:
		|\mathbf k|
		<
		\left(\frac{4g}{W_b''(0)}\right)^{1/4}
		\right\}.
	\]
	Since \(\Lambda^*\) is discrete, this bounded set contains only finitely
	many lattice points. Hence \(\Sigma(U)\) is finite.
\end{proof}

\subsection{Kernel of the linearization}

For \(\mathbf k\in\Sigma(U)\), the pressure mode is written as
\begin{align*}
	\mathscr P_{1,\mathbf k}(z)
	=
	a_{\mathbf k}Q_{\mathbf k}(z),
\end{align*}
where
\begin{align*}
	Q_{\mathbf k}(z)
	=
	\frac{\alpha^2U(0)^2}{q_{\mathbf k}(0)}
	\exp\left(
	\int_0^z q_{\mathbf k}(t)\,dt
	\right).
\end{align*}
Here \(\mathbf k=(\alpha,\beta)\). By construction,
\[
Q_{\mathbf k}'(z)
=
Q_{\mathbf k}(z)q_{\mathbf k}(z),
\]
and
\[
Q_{\mathbf k}(0)
=
\frac{\alpha^2U(0)^2}{q_{\mathbf k}(0)}.
\]
Using the dispersion relation \eqref{eq:dispersion}, we also have
\[
Q_{\mathbf k}(0)
=
g-\frac{W_b''(0)}4|\mathbf k|^4.
\]

For later use, we define
\begin{align}
	\Theta(z)
	&:=
	1+\frac{z}{d},\\
	S_{\mathbf k}(z)
	&:=
	Q_{\mathbf k}(z)q_{\mathbf k}(z)
	-
	\alpha^2U(z)^2\Theta(z),
	\label{eq:Sk_def}
\end{align}
then
\[
S_{\mathbf k}(-d)=0,
\qquad
S_{\mathbf k}(0)=0.
\]
Indeed, \(q_{\mathbf k}(-d)=0\) and \(\Theta(-d)=0\), while
\[
S_{\mathbf k}(0)
=
Q_{\mathbf k}(0)q_{\mathbf k}(0)
-
\alpha^2U(0)^2
=
0.
\]

\begin{theorem}\label{kernel}
	The kernel of the linearized problem consists of functions of the form
	\begin{align}
		\eta_1(x,y)
		&=
		a_0
		+
		\sum_{\mathbf k\in\Sigma(U)}
		a_{\mathbf k}\textup{e}^{i\mathbf k\cdot(x,y)},
		\label{eq:kernel_eta}\\
		\mathscr P_1(x,y,z)
		&=
		ga_0
		+
		\sum_{\mathbf k\in\Sigma(U)}
		a_{\mathbf k}Q_{\mathbf k}(z)\textup{e}^{i\mathbf k\cdot(x,y)},\\
		\mathbf u_1(x,y,z)
		&=
		h_0(y,z)\mathbf e_1
		+
		\sum_{\mathbf k\in\Sigma(U)}
		a_{\mathbf k}
		\begin{pmatrix}
			X_{\mathbf k}(z)\\
			V_{\mathbf k}(z)\\
			W_{\mathbf k}(z)
		\end{pmatrix}
		\textup{e}^{i\mathbf k\cdot(x,y)},
		\label{eq:kernel_u}
	\end{align}
	where 
	\begin{align*}
		V_{\mathbf k}(z)
		&=
		-\frac{\beta Q_{\mathbf k}(z)}{\alpha U(z)},\\
		W_{\mathbf k}(z)
		&=
		\frac{i}{\alpha U(z)}
		\left[
		Q_{\mathbf k}(z)q_{\mathbf k}(z)
		-
		\alpha^2U(z)^2\Theta(z)
		\right]
		=
		\frac{iS_{\mathbf k}(z)}{\alpha U(z)},\\
		X_{\mathbf k}(z)
		&=
		-\frac{Q_{\mathbf k}(z)}{U(z)}
		-
		\frac{U'(z)}{\alpha^2U(z)^2}
		\left[
		Q_{\mathbf k}(z)q_{\mathbf k}(z)
		-
		\alpha^2U(z)^2\Theta(z)
		\right]
		+
		\frac{U(z)}{d}
	\end{align*}
	for \(\mathbf k=(\alpha,\beta)\in\Sigma(U)\). The coefficients \(a_0\) and \(a_{\mathbf k}\) are real and satisfy the
	\((+)\)-symmetry condition
	\[
	a_{-\alpha,\beta}
	=
	a_{\alpha,-\beta}
	=
	a_{\alpha,\beta},
	\]
	and \(h_0\) is an arbitrary \(\lambda_2\)-periodic even function of \(y\).
\end{theorem}

\begin{proof}
	The representations of \(\eta_1\) and \(\mathscr P_1\) follow directly
	from the Fourier analysis and from the definition of \(Q_{\mathbf k}\).
	By Proposition \ref{prop:finite_modes}, every \(\mathbf k\in\Sigma(U)\)
	satisfies \(\alpha\neq0\), because the case \(\alpha=0\) would make the
	right-hand side of \eqref{eq:dispersion} equal to zero, contradicting
	\(q_{\mathbf k}(0)>0\).
	
	We now derive the velocity coefficients. Fix
	\(\mathbf k=(\alpha,\beta)\in\Sigma(U)\). The linearized momentum equation \eqref{eq:linear_momentum}
	for this Fourier mode can be written as
	\begin{align}
		U'w_{1,\mathbf k}\mathbf e_1
		+
		i\alpha U
		\begin{pmatrix}
			u_{1,\mathbf k}\\
			v_{1,\mathbf k}\\
			w_{1,\mathbf k}
		\end{pmatrix}
		+
		U^2i\alpha
		(i\alpha,i\beta,\partial_z)
		\times
		\left(
		\varphi_{1,\mathbf k}\mathbf e_2
		\right)
		+
		(i\alpha,i\beta,\partial_z)\mathscr P_{1,\mathbf k}
		=
		0,
		\label{eq:linear_momentum_fourier}
	\end{align}
	where
	\[
	\varphi_{1,\mathbf k}(z)
	=
	a_{\mathbf k}\Theta(z).
	\]
	Since
	\[
	(i\alpha,i\beta,\partial_z)
	\times
	\left(
	\varphi_{1,\mathbf k}\mathbf e_2
	\right)
	=
	\left(
	-\varphi_{1,\mathbf k}',
	0,
	i\alpha\varphi_{1,\mathbf k}
	\right),
	\]
	the \(y\)-component of \eqref{eq:linear_momentum_fourier} gives
	\[
	i\alpha Uv_{1,\mathbf k}
	+
	i\beta\mathscr P_{1,\mathbf k}
	=
	0,
	\]
	then 
	\[
	v_{1,\mathbf k}
	=
	-\frac{\beta\mathscr P_{1,\mathbf k}}{\alpha U}
	=
	-a_{\mathbf k}
	\frac{\beta Q_{\mathbf k}}{\alpha U}.
	\]
	
	The \(z\)-component gives
	\[
	i\alpha Uw_{1,\mathbf k}
	-
	\alpha^2U^2\varphi_{1,\mathbf k}
	+
	\mathscr P_{1,\mathbf k}'
	=
	0,
	\]
	then
	\[
	w_{1,\mathbf k}
	=
	\frac{i}{\alpha U}
	\left(
	\mathscr P_{1,\mathbf k}'
	-
	\alpha^2U^2\varphi_{1,\mathbf k}
	\right).
	\]
	Using
	\[
	\mathscr P_{1,\mathbf k}
	=
	a_{\mathbf k}Q_{\mathbf k},
	\qquad
	\mathscr P_{1,\mathbf k}'
	=
	a_{\mathbf k}Q_{\mathbf k}q_{\mathbf k},
	\qquad
	\varphi_{1,\mathbf k}
	=
	a_{\mathbf k}\Theta,
	\]
	we obtain
	\[
	w_{1,\mathbf k}
	=
	a_{\mathbf k}
	\frac{iS_{\mathbf k}}{\alpha U}.
	\]
	
	Finally, the \(x\)-component of \eqref{eq:linear_momentum_fourier} gives
	\[
	U'w_{1,\mathbf k}
	+
	i\alpha Uu_{1,\mathbf k}
	-
	i\alpha U^2\varphi_{1,\mathbf k}'
	+
	i\alpha\mathscr P_{1,\mathbf k}
	=
	0,
	\]
	then
	\[
	u_{1,\mathbf k}
	=
	-\frac{\mathscr P_{1,\mathbf k}}{U}
	-
	\frac{U'}{\alpha^2U^2}
	\left(
	\mathscr P_{1,\mathbf k}'
	-
	\alpha^2U^2\varphi_{1,\mathbf k}
	\right)
	+
	U\varphi_{1,\mathbf k}'.
	\]
	Since
	\[
	\varphi_{1,\mathbf k}'=\frac{a_{\mathbf k}}{d},
	\]
	which gives
	\[
	u_{1,\mathbf k}
	=
	a_{\mathbf k}X_{\mathbf k}.
	\]
	
	The boundary conditions for the vertical component follow from
	\eqref{eq:Sk_def}. Indeed,
	\[
	S_{\mathbf k}(-d)=S_{\mathbf k}(0)=0,
	\]
	and 
	\[
	W_{\mathbf k}(-d)=W_{\mathbf k}(0)=0.
	\]
	The zero mode contributes \(a_0\) to \(\eta_1\), \(ga_0\) to
	\(\mathscr P_1\), and leaves the shear-direction component
	\(h_0(y,z)\mathbf e_1\) free. Summing over all Fourier modes gives
	\eqref{eq:kernel_eta}--\eqref{eq:kernel_u}.
\end{proof}

\section{Solvability at the quadratic level}\label{sect5}

 In this section, we derive the second-order solvability condition by taking
the first component of the curl of the momentum equation and averaging over one period in the \(x\)-direction.

Recall that
\[
\varphi_1(z,x,y)
=
\left(1+\frac{z}{d}\right)\eta_1(x,y)
=
\Theta(z)\eta_1(x,y),
\qquad
\Theta(z):=1+\frac{z}{d}.
\]
We split the first-order velocity into the flattening part and the reduced
velocity by writing
\begin{align*}
	\mathbf u_1
	=
	-\nabla\times\left(U\varphi_1\mathbf e_2\right)
	+
	\mathbf V_1,
\end{align*}
where
\[
\mathbf V_1=(V_1,V_2,V_3).
\]
With this notation,  Eqs. \eqref{eq:linear_momentum}-\eqref{incompress} become
\begin{align}
	V_3U'\mathbf e_1
	+
	U\partial_x\mathbf V_1
	+
	\nabla\mathscr P_1
	=
	0,
	\label{eq:V1_linear_momentum}
\end{align}
and
\begin{align*}
	\nabla\cdot\mathbf V_1=0.
\end{align*}

At order \(\varepsilon^2\), the momentum equation \eqref{eq:flattened_momentum} obtained from
\[
\mathcal M^\top(\mathbf u\cdot\nabla)(\mathcal M\mathbf u)
+
\nabla\mathscr P
=
0
\]
has the form
\begin{align}
	&
	(\mathbf u_0\cdot\nabla)
	\left(
	\mathbf u_2+\mathcal M_1\mathbf u_1+\mathcal M_2\mathbf u_0
	\right)
	+
	(\mathbf u_1\cdot\nabla)
	\left(
	\mathbf u_1+\mathcal M_1\mathbf u_0
	\right)
	+
	(\mathbf u_2\cdot\nabla)\mathbf u_0
	\nonumber\\
	&\quad
	+
	\mathcal M_1^\top
	\left[
	(\mathbf u_0\cdot\nabla)
	\left(
	\mathbf u_1+\mathcal M_1\mathbf u_0
	\right)
	+
	(\mathbf u_1\cdot\nabla)\mathbf u_0
	\right]
	+
	\nabla\mathscr P_2
	=
	0.
	\label{eq:second_order_momentum}
\end{align}
Here \(\mathbf u_0=U(z)\mathbf e_1\), and \(\mathcal M_1,\mathcal M_2\) are
the first and second variations of the matrix \(\mathcal M\). The exact
form of \(\mathcal M_2\) will not be needed.

\begin{lemma}\label{lem:quadratic_cancellation}
Let
\[
	\mathbf u_1
	=
	-\nabla\times\left(U\varphi_1\mathbf e_2\right)
	+
	\mathbf V_1,
	\qquad
	\mathbf V_1=(V_1,V_2,V_3),
\]
then the order-\(\varepsilon^2\) equation \eqref{eq:second_order_momentum} implies the necessary condition
\begin{align}
	\int_0^{\lambda_1}
	\left(
	\nabla\times
	\left[
	(\mathbf V_1\cdot\nabla)\mathbf V_1
	\right]
	\right)_1
	\,dx
	=
	0,
	\label{eq:curl_condition_V}
\end{align}
or
\begin{align}
	\int_0^{\lambda_1}
	\left[
	(\partial_y^2-\partial_z^2)(V_2V_3)
	+
	\partial_y\partial_z
	\left(
	V_3^2-V_2^2
	\right)
	\right]\,dx=0.
	\label{eq:second_order_condition}
\end{align}
\end{lemma}

\begin{proof}  The proof is similar to that in \cite{SVVW2026}. For completeness, we include the details.
Let
\[
	F:=U\varphi_1,
\]
since
\[
	-\nabla\times(F\mathbf e_2)
	=
	(F_z,0,-F_x),
\]
then we have
\[
	\mathbf u_1=(F_z,0,-F_x)+\mathbf V_1.
\]
Moreover, together with
\[
	\mathcal M_1
	=
	\mathbf e_3\otimes\nabla\varphi_1
	-
	\varphi_{1,z}I,
	\qquad
	\mathbf u_0=U\mathbf e_1,
\]
one obtains
\[
	\mathbf u_1+\mathcal M_1\mathbf u_0
	=
	U'\varphi_1\mathbf e_1+\mathbf V_1,
\]
where
\[
	\mathcal M_1\mathbf u_0
	=
	-U\varphi_{1,z}\mathbf e_1
	+
	U\varphi_{1,x}\mathbf e_3,
\]
and
\[
	F_z=U'\varphi_1+U\varphi_{1,z},
	\qquad
	F_x=U\varphi_{1,x}.
\]

We now expand the quadratic terms in \eqref{eq:second_order_momentum}. First,
\begin{align*}
& (\mathbf u_1\cdot\nabla)
\left(
\mathbf u_1+\mathcal M_1\mathbf u_0
\right) \\
&=
\left[
F_z\partial_x(U'\varphi_1)
-
F_x\partial_z(U'\varphi_1)
+
(\mathbf V_1\cdot\nabla)(U'\varphi_1)
\right]\mathbf e_1 \\
&\quad
+
F_z\partial_x\mathbf V_1
-
F_x\partial_z\mathbf V_1
+
(\mathbf V_1\cdot\nabla)\mathbf V_1
\end{align*}
and
\[
	(\mathbf u_2\cdot\nabla)\mathbf u_0
	=
	w_2U'\mathbf e_1.
\]
Moreover, we have
\[
(\mathbf u_0\cdot\nabla)
\left(
\mathbf u_1+\mathcal M_1\mathbf u_0
\right)
+
(\mathbf u_1\cdot\nabla)\mathbf u_0
=
U\partial_x\mathbf V_1+V_3U'\mathbf e_1,
\]
because the terms involving \(F_x=U\varphi_{1,x}\) cancel. Hence, we obtain
\begin{align*}
&\mathcal M_1^\top
\left[
(\mathbf u_0\cdot\nabla)
\left(
\mathbf u_1+\mathcal M_1\mathbf u_0
\right)
+
(\mathbf u_1\cdot\nabla)\mathbf u_0
\right] \\
&=
-\varphi_{1,z}U'V_3\mathbf e_1
+
U\partial_xV_3\nabla\varphi_1
-
U\varphi_{1,z}\partial_x\mathbf V_1.
\end{align*}

Now take the first component of the curl of
\eqref{eq:second_order_momentum} and integrate over one \(x\)-period. The
pressure term disappears because \(\nabla\times\nabla\mathscr P_2=0\). The term
\[
(\mathbf u_0\cdot\nabla)
\left(
\mathbf u_2+\mathcal M_1\mathbf u_1+\mathcal M_2\mathbf u_0
\right)
=
U\partial_x
\left(
\mathbf u_2+\mathcal M_1\mathbf u_1+\mathcal M_2\mathbf u_0
\right)
\]
gives zero after integration in \(x\), since it is a total \(x\)-derivative. The
term \((\mathbf u_2\cdot\nabla)\mathbf u_0= w_2U'\mathbf e_1\) has zero first
curl component. The same observation removes the \(\mathbf e_1\)-component of
\((\mathbf u_1\cdot\nabla)(\mathbf u_1+\mathcal M_1\mathbf u_0)\).

Therefore the only terms left, besides
\[
\left(
\nabla\times
\left[
(\mathbf V_1\cdot\nabla)\mathbf V_1
\right]
\right)_1,
\]
are
\begin{align*}
\mathcal R
&:=
\partial_y
\left(
F_z\partial_xV_3
-
F_x\partial_zV_3
\right)
-
\partial_z
\left(
F_z\partial_xV_2
-
F_x\partial_zV_2
\right) \\
&\quad
+
\partial_z
\left(
U\varphi_{1,z}\partial_xV_2
-
U\varphi_{1,y}\partial_xV_3
\right).
\end{align*}
We claim that
\[
	\int_0^{\lambda_1}\mathcal R\,dx=0.
\]
Indeed, by integration by parts in the periodic variable \(x\), we have
\[
\int_0^{\lambda_1}
\partial_y
\left(
F_z\partial_xV_3
-
F_x\partial_zV_3
\right)
\,dx
=
\partial_z
\int_0^{\lambda_1}
\partial_y
\left(
F\partial_xV_3
\right)
\,dx,
\]
and 
\[
\int_0^{\lambda_1}
\partial_z
\left(
F_z\partial_xV_2
-
F_x\partial_zV_2
\right)
\,dx
=
\partial_z
\int_0^{\lambda_1}
\partial_z
\left(
F\partial_xV_2
\right)
\,dx.
\]
Thus, we have
\begin{align*}
\int_0^{\lambda_1}\mathcal R\,dx
&=
\partial_z
\int_0^{\lambda_1}
\Big[
\partial_y(F\partial_xV_3)
-
\partial_z(F\partial_xV_2)
+
U\varphi_{1,z}\partial_xV_2
-
U\varphi_{1,y}\partial_xV_3
\Big]dx.
\end{align*}
Since \(F=U\varphi_1\), the expression inside the bracket becomes
\[
\varphi_1
\left[
\partial_y(U\partial_xV_3)
-
\partial_z(U\partial_xV_2)
\right]
=
\varphi_1
\left(
\nabla\times
(U\partial_x\mathbf V_1)
\right)_1.
\]
Using the reduced first-order equation \eqref{eq:V1_linear_momentum}, that is 
\[
	U\partial_x\mathbf V_1
	=
	-
	V_3U'\mathbf e_1
	-
	\nabla\mathscr P_1,
\]
we get
\[
\left(
\nabla\times
(U\partial_x\mathbf V_1)
\right)_1
=
0.
\]
Hence \(\int_0^{\lambda_1}\mathcal R\,dx=0\), and therefore
\eqref{eq:curl_condition_V} follows.

It remains only to rewrite this condition using \(\nabla\cdot\mathbf V_1=0\). We have
\[
\left(
\nabla\times
\left[
(\mathbf V_1\cdot\nabla)\mathbf V_1
\right]
\right)_1
=
\partial_y(\mathbf V_1\cdot\nabla V_3)
-
\partial_z(\mathbf V_1\cdot\nabla V_2).
\]
Since \(\partial_xV_1=-(\partial_yV_2+\partial_zV_3)\), periodic integration by
parts in \(x\) gives
\[
\int_0^{\lambda_1}V_1\partial_xV_3\,dx
=
\int_0^{\lambda_1}(\partial_yV_2+\partial_zV_3)V_3\,dx,
\]
and
\[
\int_0^{\lambda_1}V_1\partial_xV_2\,dx
=
\int_0^{\lambda_1}(\partial_yV_2+\partial_zV_3)V_2\,dx.
\]
Consequently, one has
\begin{align*}
0
&=
\int_0^{\lambda_1}
\left(
\nabla\times
\left[
(\mathbf V_1\cdot\nabla)\mathbf V_1
\right]
\right)_1
\,dx \\
&=
\int_0^{\lambda_1}
\left[
\partial_y^2(V_2V_3)
-
\partial_z^2(V_2V_3)
+
\partial_y\partial_z(V_3^2-V_2^2)
\right]dx,
\end{align*}
which proves \eqref{eq:second_order_condition}.
\end{proof}

\section{Proof of Theorem~\ref{nn}}\label{sect6}

In this section, we insert the reduced kernel representation into the
quadratic solvability condition and use the resulting identity to prove that
all transverse leading-order Fourier modes vanish.
\begin{proof}[Proof of Theorem~\ref{nn}] By Theorem~\ref{kernel}, the first-order perturbation has the kernel representation \eqref{eq:kernel_eta}--\eqref{eq:kernel_u}.
By Lemma~\ref{lem:quadratic_cancellation}, the second-order solvability
condition is \eqref{eq:second_order_condition}. We now insert the Fourier
representation of \(\mathbf V_1\). From
\eqref{eq:V1_linear_momentum}, for
\(\mathbf k=(\alpha,\beta)\in\Sigma(U)\), we have
\begin{align*}
	V_{2,\mathbf k}
	&=
	-a_{\mathbf k}
	\frac{\beta Q_{\mathbf k}}{\alpha U},\\
	V_{3,\mathbf k}
	&=
	a_{\mathbf k}
	\frac{iQ_{\mathbf k}q_{\mathbf k}}{\alpha U}.
\end{align*}
Define
\[
\Sigma(U)^+
=
\left\{
\mathbf k=(\alpha,\beta)\in\Sigma(U):
\alpha>0,\ \beta>0
\right\}.
\]
Thus, using the \((+)\)-symmetry and the reality condition,
\begin{align}
	V_2
	&=
	4
	\sum_{\mathbf k\in\Sigma(U)^+}
	\frac{a_{\mathbf k}\beta Q_{\mathbf k}}{\alpha U}
	\sin(\alpha x)\sin(\beta y),
	\label{eq:V2_real_form}\\
	V_3
	&=
	-2
	\sum_{\substack{\alpha\mathbf e_1\in\Sigma(U)\\ \alpha>0}}
	\frac{a_{\alpha\mathbf e_1}Q_{\alpha\mathbf e_1}
		q_{\alpha\mathbf e_1}}
	{\alpha U}
	\sin(\alpha x)
	\nonumber\\
	&\quad
	-
	4
	\sum_{\mathbf k\in\Sigma(U)^+}
	\frac{a_{\mathbf k}Q_{\mathbf k}q_{\mathbf k}}
	{\alpha U}
	\sin(\alpha x)\cos(\beta y),
	\label{eq:V3_real_form}
\end{align}

We now project \eqref{eq:second_order_condition} onto transverse Fourier
modes. If all coefficients \(a_{\mathbf k}\) with \(\beta>0\) vanish, then
there is nothing to prove. Otherwise, let
\[
\beta_*
:=
\max
\left\{
\beta>0:
\text{there exists } \alpha>0
\text{ such that }
(\alpha,\beta)\in\Sigma(U)
\text{ and }
a_{\alpha,\beta}\neq0
\right\}.
\]
Projecting \eqref{eq:second_order_condition} onto the mode
\(\sin(2\beta_* y)\), and using the orthogonality of the \(x\)-Fourier
modes, only the self-interactions of modes
\[
\mathbf k=(\alpha,\beta_*)
\]
contribute. A direct calculation using
\eqref{eq:V2_real_form}--\eqref{eq:V3_real_form} gives
\begin{align*}
	0
	=
	-\frac{8U'}{U^3}
	\sum_{\substack{
			\mathbf k=(\alpha,\beta_*)\in\Sigma(U)^+}}
	\frac{
		a_{\mathbf k}^2\beta_*^2
	}{
		\alpha^2
	}
	Q_{\mathbf k}^2
	\left(
	\alpha^2-\beta_*^2+q_{\mathbf k}^2
	\right).
\end{align*}
which is equivalent to
\begin{align}
	U'f_*=0
	\quad\text{on }[-d,0],
	\label{eq:Uf_zero}
\end{align}
where
\begin{align*}
	f_*(z)
	:=
	\sum_{\substack{
			\mathbf k=(\alpha,\beta_*)\in\Sigma(U)^+}}
	\frac{
		a_{\mathbf k}^2\beta_*^2
	}{
		\alpha^2
	}
	Q_{\mathbf k}(z)^2
	\left(
	\alpha^2-\beta_*^2+q_{\mathbf k}(z)^2
	\right).
\end{align*}
Using
\[
Q_{\mathbf k}'=Q_{\mathbf k}q_{\mathbf k}
\]
and the Riccati equation
\[
q_{\mathbf k}'
=
2\frac{U'}{U}q_{\mathbf k}
+
|\mathbf k|^2
-
q_{\mathbf k}^2,
\]
we compute
\begin{align}
	f_*'(z)
	=
	4
	\sum_{\substack{
			\mathbf k=(\alpha,\beta_*)\in\Sigma(U)^+}}
	\frac{
		a_{\mathbf k}^2\beta_*^2
	}{
		\alpha^2
	}
	Q_{\mathbf k}(z)^2
	q_{\mathbf k}(z)
	\left(
	\alpha^2
	+
	\frac{U'(z)}{U(z)}q_{\mathbf k}(z)
	\right).
	\label{eq:fstar_derivative}
\end{align}

We now prove that the existence of a nonzero coefficient with
\(\beta>0\) forces \(U\) to be constant. Let
\[
V
:=
\left\{
z\in(-d,0):
U'(z)=0
\right\}.
\]
First, \(V\) is nonempty. Indeed, since
\[
q_{\mathbf k}(-d)=0,
\]
continuity implies that there exists a small interval
\[
I_\delta=(-d,-d+\delta)
\]
such that
\[
\alpha^2
+
\frac{U'}{U}q_{\mathbf k}
\ge
\frac12\alpha^2
\]
on \(I_\delta\), for all finitely many modes appearing in
\eqref{eq:fstar_derivative}. Hence
\[
f_*'(z)>0
\quad\text{on }I_\delta,
\]
because at least one coefficient \(a_{\mathbf k}\) with
\(\beta=\beta_*\) is nonzero. Therefore \(f_*\) is not identically zero
on \(I_\delta\). By \eqref{eq:Uf_zero}, this implies
\[
U'=0
\quad\text{on a nonempty subinterval of }I_\delta.
\]
Thus \(V\neq\emptyset\).

Next, let \(z_0\in V\). Then \(U'(z_0)=0\), and
\eqref{eq:fstar_derivative} gives
\[
f_*'(z_0)
=
4
\sum_{\substack{
		\mathbf k=(\alpha,\beta_*)\in\Sigma(U)^+}}
\frac{
	a_{\mathbf k}^2\beta_*^2
}{
	\alpha^2
}
Q_{\mathbf k}(z_0)^2
q_{\mathbf k}(z_0)
>0,
\]
where we have used \(q_{\mathbf k}(z)>0\) for \(z\in(-d,0]\). Hence
\(f_*\) is nonzero in a punctured neighbourhood of \(z_0\). By
\eqref{eq:Uf_zero}, \(U'\) must vanish on that punctured neighbourhood,
and therefore, by continuity, in a full neighbourhood of \(z_0\). Thus
\(V\) is open in \((-d,0)\). Since \(V\) is also closed and nonempty, we
conclude that
\[
V=(-d,0).
\]
Therefore \(U'\equiv0\), contradicting the assumption that \(U\) is
non-constant.

Consequently, all coefficients with \(\beta>0\) vanish, that is
\[
a_{\alpha,\beta}=0
\qquad
\text{whenever }(\alpha,\beta)\in\Sigma(U)
\text{ and }\beta>0.
\]
By the \((+)\)-symmetry, the same conclusion holds for all
\(\beta\neq0\). Hence the leading-order surface and pressure modes are
independent of the transverse variable \(y\), and the transverse velocity
component vanishes:
\[
\partial_y\eta_1=0,
\qquad
\partial_y\mathscr P_1=0,
\qquad
V_2=0.
\]
Returning to the decomposition
\[
\mathbf u_1
=
-\nabla\times(U\varphi_1\mathbf e_2)
+
\mathbf V_1,
\]
we also obtain
\[
v_1=0,
\qquad
\partial_yw_1=0.
\]
It remains to remove the possible \(x\)-independent mode in the first
horizontal velocity component. By Theorem~\ref{kernel}, this mode is exactly
the term \(h_0(y,z)\mathbf e_1\). By the additional normalization assumption and the kernel representation
\eqref{eq:kernel_u},
\begin{align*}
	0
	&=
	\int_0^{\lambda_1}u_1(x,y,z)\,dx  \\
	&=
	\lambda_1 h_0(y,z)
	+
	\sum_{\mathbf k=(\alpha,\beta)\in\Sigma(U)}
	a_{\mathbf k}X_{\mathbf k}(z)\ep^{i\beta y}
	\int_0^{\lambda_1}\ep^{i\alpha x}\,dx .
\end{align*}
Since every \(\mathbf k=(\alpha,\beta)\in\Sigma(U)\) satisfies
\(\alpha\neq0\), we may write
\[
\alpha=\frac{2\pi n}{\lambda_1}
\qquad
\text{for some } n\in\mathbb Z\setminus\{0\}.
\]
Therefore, we have
\[
\int_0^{\lambda_1}\ep^{i\alpha x}\,dx
=
\frac{\ep^{i\alpha\lambda_1}-1}{i\alpha}
=
\frac{\ep^{i2\pi n}-1}{i\alpha}
=
0.
\]
Hence all resonant Fourier modes have zero \(x\)-average, and the
normalization assumption gives
\[
0
=
\int_0^{\lambda_1}u_1(x,y,z)\,dx
=
\lambda_1 h_0(y,z).
\]
Thus \(h_0(y,z)=0\). Therefore the possible \(x\)-independent mode is
absent. Since the preceding argument has already ruled out all nonzero
Fourier modes with transverse wave number \(\beta\neq0\), the remaining
first-order terms depend only on \(x\) and \(z\). Thus, we obtain
\[
\partial_y u_1
=
\partial_y v_1
=
\partial_y w_1
=
0.
\]
This completes the proof.
\end{proof}

  \section*{Acknowledgments}
  Yang was supported by the National Natural Science Foundation of China (Grant No. 12561101).   	

\section*{Data availability }
No data was used for the research described in the article.

\section*{Conflict of interest }
The authors do not have any other competing interests to declare.


\begin{thebibliography}{99}
		
\bibitem{r6}
Ahmad R, Groves M D, Nilsson D. A resonant Lyapunov centre theorem with an application to doubly periodic travelling hydroelastic waves. Journal of Nonlinear Science, 2024, 34(6): 104.


\bibitem{r19}
Blyth M G, P\v{a}r\v{a}u E I, Wang Z.  Stability of hydroelastic waves in deep water. Water Waves, 2024, 6: 169-189.


\bibitem{Bonnefoy2009}
Bonnefoy F, Meylan M H, Ferrant P. Nonlinear higher-order spectral solution for a two-dimensional moving load on ice. Journal of Fluid Mechanics, 2009, 621: 215-242.

\bibitem{r10}
Chen R M, Fan L, Walsh S, Wheeler M H. Rigidity of three-dimensional internal waves with constant vorticity. Journal of Mathematical Fluid Mechanics, 2023, 25(3): 71.

\bibitem{r7}
Constantin A. An exact solution for equatorially trapped waves. Journal of Geophysical Research: Oceans, 2012, 117(C5).

\bibitem{r9}
Constantin A. Edge waves along a sloping beach. Journal of Physics A: Mathematical and General, 2001, 34(45): 9723-9731.

\bibitem{r11}
Constantin A. Two-dimensionality of gravity water flows of constant nonzero vorticity beneath a surface wave train. European Journal of Mechanics - B/Fluids, 2011, 30(1): 12-16.

\bibitem{Forbes1986}
Forbes L K. Surface waves of large amplitude beneath an elastic sheet. Part 1. High-order series solution. Journal of Fluid Mechanics, 1986, 169: 409-428.

\bibitem{Forbes1988}
Forbes L K. Surface waves of large amplitude beneath an elastic sheet. Part 2. Galerkin solution. Journal of Fluid Mechanics, 1988, 188: 491-508.

\bibitem{Gao2014}
Gao T, Vanden-Broeck J M. Numerical studies of two-dimensional hydroelastic periodic and generalised solitary waves. Physics of Fluids, 2014, 26(8).

\bibitem{Gao2016}
Gao T, Wang Z, Vanden-Broeck J M. New hydroelastic solitary waves in deep water and their dynamics. Journal of Fluid Mechanics, 2016, 788: 469-491.

\bibitem{Guyenne2012}
Guyenne P, P\u{a}r\u{a}u E I. Computations of fully nonlinear hydroelastic solitary waves on deep water. Journal of Fluid Mechanics, 2012, 713: 307-329.

\bibitem{Guyenne2014}
Guyenne P, P\u{a}r\u{a}u E I. Finite-depth effects on solitary waves in a floating ice sheet. Journal of Fluids and Structures, 2014, 49: 242-262.


\bibitem{r8}
Henry D. On three-dimensional Gerstner-like equatorial water waves. Philosophical Transactions of the Royal Society A: Mathematical, Physical and Engineering Sciences, 2018, 376(2111): 20170088.

\bibitem{r14}
Iooss G, Plotnikov P I. Asymmetrical three-dimensional travelling gravity waves. Archive for Rational Mechanics and Analysis, 2011, 200(3): 789-880.

\bibitem{r15}
Iooss G, Plotnikov P I. Small divisor problem in the theory of three-dimensional water gravity waves. American Mathematical Society, 2009.

\bibitem{r2}
Jain A K. Review of flexible risers and articulated storage systems. Ocean engineering, 1994, 21(8): 733-750.

\bibitem{zamp}Li J, Yang S. Energy relation for hydroelastic waves. Zeitschrift für angewandte Mathematik und Physik, 2025, 76(5): 205.

\bibitem{Milewski2011}
Milewski P A, Vanden-Broeck J M, Wang Z. Hydroelastic solitary waves in deep water. Journal of Fluid Mechanics, 2011, 679: 628-640.

\bibitem{Milewski2013}
Milewski P A, Vanden-Broeck J M, Wang Z. Steady dark solitary flexural gravity waves. Proceedings of the Royal Society A, 2013, 469(2150): 20120485.

\bibitem{r3}
Nevel D E. Moving loads on a floating ice sheet. US Army Cold Regions Research \& Engineering Laboratory, 1970.

\bibitem{Parau2002}
P\u{a}r\u{a}u E, Dias F. Nonlinear effects in the response of a floating ice plate to a moving load. Journal of Fluid Mechanics, 2002, 460: 281-305.

\bibitem{PT2011}
Plotnikov P I, Toland J F. Modelling nonlinear hydroelastic waves. Philosophical Transactions of the Royal Society A: Mathematical, Physical and Engineering Sciences, 2011, 369(1947): 2942-2956.

\bibitem{r13}
Reeder J, Shinbrot M. Three-dimensional, nonlinear wave interaction in water of constant depth. Nonlinear Analysis: Theory, Methods \& Applications, 1981, 5(3): 303-323.

\bibitem{r1}
Seth D S. Internal doubly periodic gravity-capillary waves with vorticity. SIAM Journal on Mathematical Analysis, 2024, 56(6): 7545-7575.

\bibitem{SVW2024}
Seth D S, Varholm K, Wahl\'en E. Symmetric doubly periodic gravity-capillary waves with small vorticity. Advances in Mathematics, 2024, 447: 109683.

\bibitem{SVVW2026}
Seth D S, Varholm K, Wahl\'en E, Weber J. Rigidity of symmetric doubly-periodic water waves near shear flows. Nonlinearity, 2026, 39(3): 035015.

\bibitem{r4}
Squire V A, Robinson W H, Langhorne P J, Haskell T G. Vehicles and aircraft on floating ice. Nature, 1988, 333(6169): 159-161.

\bibitem{r5}
Takizawa T. Deflection of a floating sea ice sheet induced by a moving load. Cold Regions Science and Technology, 1985, 11(2): 171-180.

\bibitem{Toland2007}
Toland J. Heavy hydroelastic travelling waves. Proceedings of the Royal Society A, 2007, 463(2085): 2371-2397.

\bibitem{Toland2008}
Toland J F. Steady periodic hydroelastic waves. Archive for Rational Mechanics and Analysis, 2008, 189(2): 325-362.

\bibitem{VandenBroeck2011}
Vanden-Broeck J M, P\u{a}r\u{a}u E I. Two-dimensional generalized solitary waves and periodic waves under an ice sheet. Philosophical Transactions of the Royal Society A, 2011, 369(1947): 2957-2972.

\bibitem{r12}
Wahl\'en E. Non-existence of three-dimensional travelling water waves with constant non-zero vorticity. Journal of Fluid Mechanics, 2014, 746: R2.

\bibitem{R1}
Yang J. Cubic lifespan of two-dimensional hydroelastic waves. Communications in Mathematical Sciences, 2024, 23(1): 259-277.
	\end{thebibliography}
\end{document}